\documentclass{amsart}
\usepackage{amsmath,amssymb,amsthm}
\usepackage{graphics,graphicx}
\usepackage{enumerate}
\usepackage{fancyhdr}
\usepackage[T1]{fontenc}
\usepackage{rotating}
\usepackage{MnSymbol}
\usepackage{tikz-cd}
\usepackage{tikz}

\usepackage[labelsep=none]{caption}
\usepackage{float}

\newcounter{dummy}

\newtheorem{theorem}[dummy]{Theorem}
\newtheorem*{theorem*}{Theorem}
\newtheorem{conjecture}[dummy]{Conjecture}
\newtheorem{lemma}[dummy]{Lemma}
\newtheorem{cor}[dummy]{Corollary}
\newtheorem*{cor*}{Corollary}

\theoremstyle{remark}
\newtheorem{remark}[dummy]{Remark}

\newcommand{\Q}{\mathbf{q}}

\date{}

\title[The weighted Singer conjecture in dimensions three and four]{The weighted Singer conjecture for Coxeter groups in dimensions three and four}
\author{Wiktor J. Mogilski}
\address{Department of Mathematical Sciences, University of Wisconsin-Milwaukee, Milwaukee, WI 53211, USA}
\email{mogilski@uwm.edu}

\begin{document}

\begin{abstract}
Given a Coxeter system $(W,S)$ there is a contractible simplicial complex $\Sigma$ called the Davis complex on which $W$ acts properly and cocompactly. In \cite{Dymara}, the weighted $L^2$-(co)homology groups of $\Sigma$ were defined, and in \cite{DDJO}, the Singer conjecture for Coxeter groups was appropriately formulated for weighted $L^2$-(co)homology theory. In this article, we prove the weighted version of the Singer conjecture in dimension three under the assumption that the nerve of the Coxeter group is not dual to a hyperbolic simplex, and in dimension four under additional restrictions. We then prove a general version of the conjecture where the nerve of the Coxeter group is assumed to be a flag triangulation of a $3$-manifold.
\end{abstract}
\maketitle

Given a Coxeter system $(W,S)$ with corresponding nerve $L$, Davis constructs a contractible simplicial complex $\Sigma_L$ on which $W$ acts properly and cocompactly (for details see \cite{Davis,Davis1}). Given an $S$-tuple $\Q=(q_s)_{s\in S}$ of positive real numbers satisfying $q_s=q_{s'}$ if $s$ and $s'$ are conjugate in $W$, one defines the weighted $L^2$-(co)homology spaces $L^2_\Q H_k(\Sigma_L)$ (see \cite{DDJO}). They are special in the sense that they admit a notion of dimension: one can attach a nonnegative real number to each of the Hilbert spaces $L^2_\Q H_k(\Sigma_L)$ called the von Neumann dimension. Hence one can define weighted $L^2$-Betti numbers, denoted by $L^2_\Q b_k(\Sigma_L$). More details on this theory can be found in \cite{Davis, DDJO, Dymara}. When $\Q=\mathbf{1}$, the groups $L^2_\mathbf{1} H_k(\Sigma_L)$ are the $L^2$-(co)homology groups of $\Sigma_L$ considered in \cite{DO}. There, the classical Singer conjecture for manifolds is reformulated for Coxeter groups, using the fact that if $L$ is a triangulation of the $(n-1)$-sphere, then $\Sigma_L$ is an $n$-manifold.

\begin{conjecture}[Singer Conjecture]
\label{conj:singer}
Suppose that $L$ is a triangulation of $S^{n-1}$. Then $$L_\mathbf{1}^2H_k(\Sigma_L)=0 \text{ for } k\neq \frac{n}{2}.$$
\end{conjecture}

Conjecture \ref{conj:singer} is known for elementary reasons for $n\leq 2$ and holds by a result of Lott and L\"{u}ck \cite{LottLuck}, in conjunction with the validity of the Geometrization Conjecture for $3$-manifolds \cite{Perelman}, for $n=3$. It was proved by Davis--Okun \cite{DO} for the case where $W$ is right-angled and $n\leq 4$, and furthermore it was shown that Conjecture \ref{conj:singer} for $n$ odd implies Conjecture \ref{conj:singer} for $n$ even, under the assumption that $W$ is right-angled. It was later proved for the case where $W$ is an even Coxeter group and $n=4$ by Schroeder \cite{Schroeder1}, under the restriction that $L$ is a flag complex. Recently Okun--Schreve \cite[Theorem 4.9]{OS} gave a proof of Conjecture \ref{conj:singer} for the case $n=4$, so now Conjecture \ref{conj:singer} is known in full generality for dimensions $n\leq 4$.

In \cite{DDJO}, the Singer conjecture for Coxeter groups was formulated for weighted $L^2$-(co)homology:

\begin{conjecture}[Weighted Singer Conjecture]
\label{conj:weightedsinger}
Suppose that $L$ is a triangulation of $S^{n-1}$. Then $$L_\Q^2H_k(\Sigma_L)=0 \text{ for } k>\frac{n}{2} \text{ and } \Q\leq \mathbf{1}.$$
\end{conjecture}

By weighted Poincar\'{e} duality, this is equivalent to the conjecture that if $\Q\geq \mathbf{1}$ and $k<\frac{n}{2}$, then $L_\Q^2H_k(\Sigma_L)$ vanishes. Conjecture \ref{conj:weightedsinger} is known for elementary reasons for $n\leq 2$, and in \cite{DDJO}, it was proved for the case where $W$ is right-angled and $n\leq 4$, and furthermore it was shown that Conjecture \ref{conj:weightedsinger} for $n$ odd implies Conjecture \ref{conj:weightedsinger} for $n$ even (also under the assumption that $W$ is right-angled). In this article we prove Conjecture \ref{conj:weightedsinger} for $n=3,4$, under some additional restrictions on the nerve $L$.

\begin{theorem}
\label{thm:weightedSingerdim3}
Suppose that the nerve $L$ of a Coxeter group is a triangulation of $S^2$ not dual to a hyperbolic $3$-simplex. Then $$L_\Q^2H_k(\Sigma_L)=0 \text{ for } k>1 \text{ and } \Q\leq \mathbf{1}.$$
\end{theorem}

Note that, in conjunction with weighted Poincar\'{e} duality and \cite[Theorem 7.1]{Dymara}, Theorem \ref{thm:weightedSingerdim3} explicitly describes the behavior of the $L^2_\Q$-(co)homology groups for all $\Q$: they are always concentrated in a single dimension. Let $\mathcal{R}$ denote the region of convergence of the growth series of the corresponding Coxeter group.

\begin{cor}
Suppose that the nerve $L$ of a Coxeter group is a triangulation of $S^2$ not dual to a hyperbolic $3$-simplex.
\begin{itemize}
  \item If $\Q\in\bar{\mathcal{R}}$, then $L_\Q^2H_\ast(\Sigma_L)$ is concentrated in dimension $0$.
  \item If $\Q\notin\mathcal{R}$ and $\Q\leq\mathbf{1}$, then $L_\Q^2H_\ast(\Sigma_L)$ is concentrated in dimension $1$.
  \item If $\Q\notin\mathcal{R}^{-1}$ and $\Q\geq\mathbf{1}$, then $L_\Q^2H_\ast(\Sigma_L)$ is concentrated in dimension $2$.
  \item If $\Q\in\bar{\mathcal{R}}^{-1}$, then $L_\Q^2H_\ast(\Sigma_L)$ is concentrated in dimension $3$.
\end{itemize}
\end{cor}

In either case one can use \cite[Corollary 3.4]{Dymara}, along with a standard computation for growth series \cite[Theorem 17.1.9]{Davis}, to explicitly compute each $L^2_\Q$-Betti number.

The Coxeter groups whose nerves are dual to hyperbolic $3$-simplices are sometimes in literature called L\'{a}nner groups, and there is only nine L\'{a}nner groups in dimension three (see \cite[Table 6.9]{Humphreys}). So there are only nine groups standing in the way of proving Conjecture \ref{conj:weightedsinger} in full generality for dimension three.

In dimension four we prove the following theorem. First, recall that a subcomplex $A$ of a simplicial complex $L$ is \emph{full} if the vertices of a simplex of $L$ lie in $A$, then the simplex must lie in $A$.

\begin{theorem}
\label{thm:singerdim4onefulllink}
Suppose that the nerve $L$ of a Coxeter group is a triangulation of $S^3$. Furthermore, suppose that there exists a vertex of $L$ such that its link is a full subcomplex of $L$ and not dual to a hyperbolic $3$-simplex. Then $$L_\Q^2H_k(\Sigma_L)=0 \text{ for } k>2 \text{ and } \Q\leq \mathbf{1}.$$
\end{theorem}

We obtain the following corollary.

\begin{cor}
Suppose that the nerve $L$ of a Coxeter group is a flag triangulation of $S^3$. Then $$L_\Q^2H_k(\Sigma_L)=0 \text{ for } k>2 \text{ and } \Q\leq \mathbf{1}.$$
\end{cor}
\begin{proof}
Since $L$ is flag, it follows that the link of every vertex is a full subcomplex of $L$. Furthermore, the link of every vertex is not the boundary of a $3$-simplex (and in particular, not dual to a $3$-simplex). Theorem \ref{thm:singerdim4onefulllink} now completes the proof.
\end{proof}

We then conclude the article by proving the following generalization of the above corollary.

\begin{theorem}
\label{thm:weighted3manifoldnerve}
Suppose that the nerve $L$ of a Coxeter group is a flag triangulation of a $3$-manifold. Then $$L_\Q^2H_k(\Sigma_L)=0 \text{ for } k>2 \text{ and } \Q\leq \mathbf{1}.$$
\end{theorem}

Consider the case where the nerve $L$ is a triangulation of the $n$-disk (in this case $\Sigma_L$ is a manifold with boundary). The following theorem will serve as one of the main ingredients in our proofs.

\begin{theorem}[{\cite[Theorem 5.16]{Mogilski}}]
\label{thm:weightedsingerfor34disks}
Suppose that the nerve $L$ of a Coxeter group is a triangulation of the $n$-disk, where $n=2,3$. If $\Q\leq \mathbf{1}$ and $k>\frac{n}{2}$, then $L_\Q^2b_k(\Sigma_L)=0$.
\end{theorem}

\subsection{A special case of Andreev's theorem.}For the purpose of this article, we use the notation $\Sigma_L=X$ whenever $\Sigma_L$ admits a $W_L$-invariant metric making it isometric to $X$. The following theorem is now a special case of Andreev's theorem \cite[Theorem 2]{Andreev}.

\begin{theorem}
\label{thm:andreevL}
Suppose that the nerve $L$ is a triangulation of $S^2$, but not the boundary of a $3$-simplex, and let $(W,S)$ be the corresponding Coxeter system. Furthermore, suppose that
\begin{itemize}
\item For every $T\subset S$, $W_T$ is not a Euclidean reflection group.
\item $W\neq W_T\times D_\infty$, where $T\subset S$ spans empty triangle in $L$ and $D_\infty$ is the infinite dihedral group.
\end{itemize}
Then $\Sigma_L=\mathbb{H}^3$.
\end{theorem}
\begin{proof}
The nerve $L$ of a Coxeter system $(W,S)$ has a natural piecewise spherical structure, and under this structure, if $s,t\in S$ are connected by an edge in $L$, then the edge has length $\pi-\pi/m_{st}$, where $(st)^{m_{st}}=1$. Hence $L$ inherits the structure of a \emph{metric flag complex} \cite[Lemma 12.3.1]{Davis}, meaning that any collection of pairwise connected edges of $L$ spans a simplex if and only if there exists a spherical simplex with the corresponding edge lengths.

Suppose now that $L$ is a triangulation of $S^2$. Let $C$ be an empty circuit in $L$ and suppose that $C$ is not the boundary of two adjacent triangles. We say that $C$ is a \emph{Euclidean circuit} if the corresponding Coxeter group $W_C$ is a Euclidean reflection group. It follows from $L$ being a metric flag complex that a Euclidean circuit is always a full subcomplex of $L$, and in particular, $W_C$ is a special subgroup of $W$. Conditions (i) and (ii) can now be translated from the conditions given in Andreev's Theorem (for details on how to derive them, see for example \cite[Section 3]{Schroeder2}).
\end{proof}

\subsection{Equidistant hypersurfaces.}

Suppose that the Coxeter group $W$ has nerve $L$ that is a triangulation of $S^2$ and that $\Sigma_L=\mathbb{H}^3$. Let $D$ denote the Davis chamber (in $\mathbb{H}^3$) and let $W_M$ be a special subgroup of $W$. We now consider the (possibly infinite) convex polytope $W_M D$ in $\mathbb{H}^3$.

For $t>0$, let $S_t$ denote the $t$-distant surface from a component $S$ of $\partial W_M D$. Then $S_t$ is a smooth surface (see \cite[Proposition II.2.2.1]{CEM}). In fact, $S_t$ is a union of pieces of which there are three types: hyperbolic, Euclidean, and spherical, each of which are the equidistant pieces from faces, edges, and vertices of $S$, respectively. The Euclidean pieces look like rectangles that are each adjacent to two hyperbolic pieces and two spherical pieces, and the spherical pieces are adjacent to Euclidean pieces.

As $W_M D$ is convex, the nearest point projection $p:\mathbb{H}^3\cup \partial\mathbb{H}^3\rightarrow W_M D$ is defined. If we fix $t>r>0$, then $p$ induces a map $p_{tr}: S_t\rightarrow  S_r$.

\begin{lemma}
\label{lemma:nearestpointquasi}
The map $p_{tr}: S_t\rightarrow  S_r$ induced by nearest point projection is  $\frac{\tanh(t)}{\tanh(r)}$-quasiconformal.
\end{lemma}
\begin{proof}
It suffices to check what $p_{tr}$ does on each of the three types of pieces. First, note that a face of $S$ is simply the intersection of $\partial W_M D$  with a hyperbolic plane in $\mathbb{H}^3$. Thus $p_{tr}$ simply scales the corresponding hyperbolic pieces on $S_t$ and $S_r$ by a constant factor. Hence $p_{tr}$ is conformal there. Similarly, the map $p_{tr}$  is conformal on the spherical pieces.

Second, we consider the Euclidean piece in $S_t$ equidistant from an edge of $S$. A Euclidean piece looks like a rectangle adjacent to two hyperbolic pieces at two parallel edges (parallel in the intrinsic Euclidean geometry), and the the map induced by nearest point projection $S_t\rightarrow S$ scales by a factor of $1/\cosh(t)$ in the direction of those edges. The other two edges of the Euclidean piece are each adjacent to a spherical piece. An edge like this is the arc of a circle with radius $t$ centered at a vertex in $S$. Thus the edge has length $\theta\sinh(t)$, where $\theta$ is the dihedral angle at the corresponding edge of $S$. Hence the map $p_{tr}$ scales by a factor of $\cosh(r)/\cosh(t)$ in the direction of the edges adjacent to the hyperbolic pieces, and scales the edges adjacent to the spherical pieces by a factor of $\sinh(r)/\sinh(t)$. Therefore $p_{tr}$ is $\frac{\tanh(t)}{\tanh(r)}$-quasiconformal on the Euclidean pieces.
\end{proof}

\subsection{Proof of Theorem \ref{thm:weightedSingerdim3}.}

Suppose that $M$ is a complete smooth Riemannian manifold. Given a a nonnegative measurable function $f:M\rightarrow [0,\infty)$, we define a new norm on the $C^\infty$ $k$-forms called the $L^2_f$ norm by $$||\omega||^2_f=\int_{M} ||\omega||^2_p f(p) dV,$$ where $||\omega||^2_p$ is the pointwise norm and $dV$ is the volume form of $M$. Let $L_{f}^2\mathcal{C}^\ast(M)$ denote the weighted $L^2$ de Rham complex defined using the $L^2_f$ norm.

\begin{lemma}
\label{lemma:middlequasi}
Let $M$ and $N$ be smooth surfaces and suppose that $\phi:M\rightarrow N$ is a $K$-quasiconformal diffeomorphism. Let $g:N\rightarrow [0,\infty)$ be the function defined by $g(p)=f(\phi^{-1}(p))$. Then for every $\omega\in L_g^2\mathcal{C}^1(N)$, we have that $$\frac{1}{K}||\omega||^2_g\leq ||\phi^\ast(\omega)||^2_f\leq K ||\omega||^2_g.$$
\end{lemma}
\begin{proof}
The pointwise norm of a $1$-form is $||\omega||_p=\sup_{\substack{x\in T_p M\\ ||x||=1}} \omega(x)$, where $T_p M$ is the tangent space of $M$ at $p$. Since $\phi$ is $K$-quasiconformal, its differential $d\phi$ maps the circle $\{x\in T_p M\mid ||x||=1\}$ to an ellipse with semi-axis $b(p)\leq a(p)$ satisfying $\frac{a(p)}{b(p)}\leq K$. Thus for any $\omega\in L_{\Q}^2\mathcal{C}^1(N)$, $$b(p)||\omega||_{\phi(p)}\leq ||\phi^\ast(\omega)||_p\leq a(p)||\omega||_{\phi(p)}.$$ Now, let $dV_M$ and $dV_N$ be the respective volume forms of $M$ and $N$. We have that $$(fdV_M)_p=\frac{(g(\phi)\phi^\ast(dV_N))_{p}}{a(p)b(p)},$$ so for $L^2_f$ norms we have

\begin{align*}
||\phi^\ast(\omega)||^2_f&=\int_{M} ||\phi^\ast(\omega)||^2_{p} f(p) dV_M\\
&\leq \int_{M} \frac{a(p)}{b(p)}||\omega||^2_{\phi(p)} g(\phi(p)) \phi^\ast(dV_N)\\
&\leq K \int_{M}||\omega||^2_{\phi(p)} g(\phi(p)) \phi^\ast(dV_N)\\
&= K\int_N ||\omega||^2_x g(x)dV_N=K ||\omega||^2_g
\end{align*}

The remaining inequality follows similarly.
\end{proof}

Suppose that the nerve $L$ of $W$ is a triangulation of $S^2$ and that $\Sigma_L=\mathbb{H}^3$. Define $f$ to be the function $f(p)=q_w$, where $w\in W_L$ is a word of shortest length such that $p\in wD$ (here $D$ is the Davis chamber). Let $L_{\Q}^2\mathcal{H}^\ast(\mathbb{H}^3)$ denote the weighted $L^2$ de Rham cohomology defined using this $f$.

Let $W_M$ be an infinite special subgroup of $W$ and let $S$ be one of the components of $\partial W_M D$. Put coordinates $(x,t)$ on $\mathbb{H}^3$ so that $t\in\mathbb{R}$ is the oriented distance from $p\in\mathbb{H}^3$ to the closest point $x\in S$. Fix $r>0$, and for $t\geq r$ let $S_t$ denote the hypersurface consisting of points of (oriented) distance $t$ from $S$.  Let $p_{tr}: S_t\rightarrow  S_r$ be the map induced by nearest point projection, and let $\phi_{tr}$ denote the inverse of $p_{tr}$. By Lemma \ref{lemma:nearestpointquasi}, $p_{tr}$ is $K(t)$-quasiconformal, with $K(t)=\frac{\tanh(t)}{\tanh(r)}$, and hence so is its inverse $\phi_{tr}: S_r\rightarrow S_t$. Let $i_r: S_r\rightarrow\mathbb{H}^3$ and $i_t: S_t\rightarrow\mathbb{H}^3$ be the inclusions. Then $i_r$ and $i_t\circ\phi_{tr}$ are properly homotopic.

We now adapt the argument after \cite[Theorem 16.10]{DDJO} to prove the following lemma.

\begin{lemma}
\label{lemma:integrationquasi}
If $\Q\geq\mathbf{1}$, then the map $i_r^\ast:L_{\Q}^2\mathcal{H}^1(\mathbb{H}^3)\rightarrow L_{\Q}^2\mathcal{H}^1(S_r)$ induced by the inclusion $i_r$ is the zero map.
\end{lemma}
\begin{proof}
Set $g(x,y)=f(x,0)$, so $f(x,y)\geq g(x,y)$, and let $\omega$ be a closed $L^2_f$ $1$-form on $\mathbb{H}^3$. We now show that the restriction $i_r^\ast(\omega)$ to $S_r$ represents the zero class in in reduced $L^2_f$-cohomology. For the remainder of the proof, we will use the notation $||[\alpha]||_g$ and $||[\alpha]||_x$ to denote the respective $L^2_g$ norm and pointwise norm of the harmonic representative of the cohomology class $[\alpha]$.

Suppose for a contradiction that $[i_r^\ast(\omega)]\neq 0$. Then $||i_r^\ast(\omega)||_g\geq ||[i_r^\ast(\omega)]||_g> 0$. By Lemma \ref{lemma:middlequasi}, it follows that $||\phi^\ast_{tr}(i_t^\ast (\omega))||^2_g\leq K(t)||i_t^\ast (\omega)||^2_g$, and since $i_r$ and $i_t\circ\phi_{tr}$ are properly homotopic, $[i_r^\ast(\omega)]=[\phi^\ast_{tr}(i_t^\ast (\omega))]$.  Therefore $$K(t)||i_t^\ast (\omega)||^2_g\geq ||[i_r^\ast(\omega)]||^2_g>0.$$ Now, $i_t^\ast (\omega)$ is just a restriction of $\omega$, so we have the pointwise inequality $||\omega||_x\geq ||i_t^\ast (\omega)||_x$. Using Fubini's Theorem, we compute

\begin{align*}
||\omega||^2_g&=\int_{\mathbb{H}^3} ||\omega||^2_{x} g(x,y) dV\geq\int^\infty_{r}\int_{S_t} ||\omega||^2_{x} g(x,y) dA dt\geq \int^\infty_{r}\int_{S_t} ||i_t^\ast (\omega)||^2_{x} g(x,y) dA dt\\
&=\int^\infty_{r}||i_t^\ast (\omega)||^2_{g} dt\geq \int^\infty_{r}\frac{\tanh(r)}{\tanh(t)}||[i_r^\ast(\omega)]||^2_g dt=\infty.
\end{align*}
Since $||\omega||_f\geq ||\omega||_g$, this contradicts the assumption that the $L^2_f$ norm of $\omega$ is finite.
\end{proof}

Suppose that $L$ is the nerve of a Coxeter group $W_L$ and that $A$ is a full subcomplex of $L$. For the proofs that follow, note that $\dim_\Q L_\Q^2H_k(W_L\Sigma_A)=L_\Q^2 b_k(\Sigma_A)$ (see \cite[pg. 352 (vi)]{Davis}).

\begin{lemma}
\label{lemma:Ldecomp}
Suppose that the nerve $L$ is a triangulation of $S^2$ and that there exists a full subcomplex $1$-sphere $M$ of $L$ that separates $L$ into two full $2$-disks $L_1$ and $L_2$ with boundary $M$. Furthermore, suppose that one of the following holds:
\begin{enumerate}[(i)]
  \item \label{lemma:Ldecomp1} $\Sigma_M=\mathbb{R}^2$.
  \item \label{lemma:Ldecomp2} $\Sigma_L=\mathbb{H}^3$.
\end{enumerate}

Then $$L_\Q^2H_k(\Sigma_L) \text{ for } k\geq 2 \text{ and } \Q\leq \mathbf{1}.$$
\end{lemma}
\begin{proof} Since $\Sigma_L$ is a $3$-manifold, it follows that $L_\Q^2 b_3(\Sigma_L)=0$ \cite[Proposition 20.4.1]{Davis}. Hence we must show that $L_\Q^2 b_2(\Sigma_L)=0$.
Consider the following Mayer--Vietoris sequence applied to $L=L_1\cup_M L_2$:

$$\begin{tikzcd}[column sep=tiny]
\dotsm\arrow{r} & L_\Q^2H_2(W_L\Sigma_{L_1})\oplus L_\Q^2H_2(W_L\Sigma_{L_2}) \arrow{r} & L_\Q^2H_2(\Sigma_L) \arrow{r} & L_\Q^2H_{1}(W_L\Sigma_M) \arrow{r} &\dotsm
\end{tikzcd}$$

By Theorem \ref{thm:weightedsingerfor34disks}, we have that $L_\Q^2H_2(W_L\Sigma_{L_1})=L_\Q^2H_2(W_L\Sigma_{L_2})=0$. If (\ref{lemma:Ldecomp1}) holds, then \cite[Corollary 14.5]{DDJO} implies that $L_\Q^2H_1(\Sigma_M)=0$, and we are done. If (\ref{lemma:Ldecomp2}) holds, we argue that the connecting homomorphism $\partial_\ast: L_\Q^2H_2(\Sigma_L)\rightarrow L_\Q^2H_{1}(W_L\Sigma_M)$ is the zero map. By \cite[Lemma 16.2]{DDJO}, we reduce the proof to showing that the map induced by inclusion $i_\ast: L_{\Q^{-1}}^2H_1(W_L\Sigma_M) \rightarrow L_{\Q^{-1}}^2H_1(\Sigma_L)$ is the zero map, and since $W_L\Sigma_M$ is a disjoint union of copies of $\Sigma_M$, it is enough to show that the restriction of $i_\ast$ to one summand $L_{\Q^{-1}}^2H_1(\Sigma_M)$ is zero.

Consider the infinite convex polytope $W_M D$, where $D$ is the Davis chamber for $W$. We have that $W_M$ acts properly and cocompactly on $W_M D$ by isometries. In particular, if $S$ is one of the components of $\partial W_M D$, then $W_M$ acts properly and cocompactly on $S$, and therefore $L_{\Q^{-1}}^2H^\ast(\Sigma_M)\cong L_{\Q^{-1}}^2\mathcal{H}^\ast(S)$. Hence we are done if we show that map $i^\ast:L_{\Q^{-1}}^2\mathcal{H}^1(\mathbb{H}^3)\rightarrow L_{\Q^{-1}}^2\mathcal{H}^1(S)$ induced by the inclusion $i:S\rightarrow\mathbb{H}^3$ is the zero map.

Fix $r>0$, and let $S_r$ be the $r$-distance surface from $S$. $S_r$ and $S$ are properly homotopy equivalent, and this equivalence induces a weak isomorphism between $L_{\Q^{-1}}^2\mathcal{H}^\ast(S)$ and $L_{\Q^{-1}}^2\mathcal{H}^\ast(S_r)$. Thus we have reduced the proof to showing that the map $i_r^\ast:L_{\Q^{-1}}^2\mathcal{H}^1(\mathbb{H}^3)\rightarrow L_{\Q^{-1}}^2\mathcal{H}^1(S_r)$ induced by the inclusion $i_r: S_r\rightarrow\mathbb{H}^3$ is the zero map, and therefore we are done by Lemma \ref{lemma:integrationquasi}.
\end{proof}

\begin{remark}
In \cite[Section 16]{DDJO} $W$ is strictly assumed to be right-angled, but the proof of \cite[Lemma 16.2]{DDJO} does not use this, as it only uses properties of weighted $L^2$-(co)homology.
\end{remark}

\begin{proof}[Proof of Theorem \ref{thm:weightedSingerdim3}]
We first suppose that $\Sigma_L=\mathbb{H}^3$. We need to find a full subcomplex $M$ of $L$ satisfying the hypothesis of Lemma \ref{lemma:Ldecomp}. First we suppose that $L$ is a flag complex. Let $v$ be a vertex of $L$ and set $M=Lk(v)$. Since $L$ is flag, $M$ is a full subcomplex of $L$, and since $L$ is a triangulation of the $2$-sphere, it follows that $M$ is a $1$-sphere, and we are done. Now suppose that $L$ is not flag. Since $L$ is not the boundary of a $3$-simplex, there exists an empty $2$-simplex in $L$. Let $M$ denote this empty $2$-simplex. Then $M$ separates $L$ into two full $2$-disks, both with boundary $M$, and we are done. We now suppose that $\Sigma_L\neq\mathbb{H}^3$ and use Theorem \ref{thm:andreevL} to perform a case-by-case analysis.
\vskip2mm
\noindent\emph{Case I: $W$ contains a Euclidean special subgroup $W_T$.} Let $M$ be the full subcomplex of $L$ corresponding to $W_T$. Then $M$ separates $L$ into two $2$-disks both with boundary $M$ and hence Lemma \ref{lemma:Ldecomp} (\ref{lemma:Ldecomp1}) implies the assertion.
\vskip2mm
\noindent\emph{Case II: $W= W_T\times D_\infty$, where $T\subset S$ spans empty triangle in $L$.} Either $\Sigma_L=\mathbb{R}^3$ or $\Sigma_L=\mathbb{H}^2\times\mathbb{R}$. In both cases we are done by the weighted K\"{u}nneth formula.
\vskip2mm
\noindent\emph{Case III: $L$ is the boundary of a $3$-simplex.} By assumption, $L$ is not dual to a hyperbolic simplex, so $\Sigma_L=\mathbb{R}^3$. Therefore we are done by \cite[Corollary 14.5]{DDJO}.
\end{proof}

\subsection{Proof of Theorem \ref{thm:singerdim4onefulllink}.}

In this case, $\Sigma_L$ is a $4$-manifold, and hence $L_\Q^2 b_4(\Sigma_L)=0$ \cite[Proposition 20.4.1]{Davis}. It remains to show that $L_\Q^2 b_3(\Sigma_L)=0$. Suppose that the nerve $L$ is a triangulation of $S^3$ and let $s\in L$ be a vertex. We make the following observations:

\begin{itemize}
  \item The nerve $L_{S-s}$ of the Coxeter system $(W_{S-s},S-s)$ is a $3$-disk.
  \item The nerve $St(s)$ of the Coxeter group $W_{St(s)}$ is a $3$-disk.
  \item The nerve $Lk(s)$ of the Coxeter group $W_{Lk(s)}$ is a $2$-sphere.
\end{itemize}

This is because the subcomplexes $St(s)$, $Lk(s)$, and $L_{S-s}$ of $L$ correspond to the closed star of the vertex $s$, link of the vertex $s$, and complement of the open star of $s$, respectively, which are all by assumption full subcomplexes of $L$.

Consider the following Mayer--Vietoris sequence:

$$\begin{tikzcd}[column sep=tiny]
\dotsm\arrow{r} & L_\Q^2H_3(W_L\Sigma_{L_{S-s}})\oplus L_\Q^2H_3(W_L\Sigma_{St(s)}) \arrow{r} & L_\Q^2H_3(\Sigma_L) \arrow{r} & L_\Q^2H_2(W_L\Sigma_{Lk(s)}) \arrow{r} &\dotsm
\end{tikzcd}$$

By Theorem \ref{thm:weightedsingerfor34disks}, $L^2_\Q b_3(\Sigma_{St(s)})=0$ and $L^2_\Q b_3(\Sigma_{L_{S-s}})=0$, and by Theorem \ref{thm:weightedSingerdim3}, $L^2_\Q b_2(\Sigma_{Lk(s)})=0$. Therefore by the above sequence, $L^2_\Q b_3(\Sigma_L)=0$.

\subsection{Proof of Theorem \ref{thm:weighted3manifoldnerve}}

Given a Coxeter system $(W,S)$ we let $\Delta$ denote the standard $(|S|-1)$-simplex with codimension-one faces indexed by $S$. Thus every $T\subset S$ corresponds to a codimension-$|T|$ face of $\Delta$, denoted by $\Delta_T$, with corresponding barycenter $v_T$. The Davis chamber $D$ is the subcomplex of the barycentric subdivision of $\Delta$ spanned by the barycenters $v_T$ with $T\in\mathcal{S}$, where $\mathcal{S}$ is the poset of subsets $T\subset S$ such that $W_T$ is finite. Set $D_T=D\cap\Delta_T$.

For each $T\in\mathcal{S}$, let $c_T$ denote the union of simplices $c\subset\Sigma_L$ such that $c\cap D_T=v_T$. The boundary of $c_T$ is then cellulated by $wc_U$, where $w\in W_T$ and $U\subset T$. With its simplicial structure, the boundary $\partial c_T$ is the Coxeter complex corresponding to the Coxeter system $(W_T,T)$, which is a sphere since $W_T$ is finite. It follows that $c_T$ and its translates are disks, which are called \emph{Coxeter cells of type $T$}. We denote $\Sigma_L$ with this decomposition into Coxeter cells by $\Sigma_{cc}$.

Now, for for $U\subset S$, set $\mathcal{S}(U):=\{T\in\mathcal{S}\mid T\subset U\}$. Define $\Sigma(U)$ to be the subcomplex of $\Sigma_{cc}$ consisting of all (closed) Coxeter cells of type $T$ with $T\in\mathcal{S}(U)$. Given $T\in\mathcal{S}(U)$, we define the following subcomplexes of $\Sigma(U)$:

\begin{align*}
\Omega_{UT}: & \hskip2mm \text{the union of closed cells of type }T', \text{ with }T'\in\mathcal{S}(U)_{\geq T},\\
\partial\Omega_{UT}: & \hskip2mm \text{the cells of }\Omega(U,T)\text{ of type } T'', \text{ with } T''\not\in\mathcal{S}(U)_{\geq T}.
\end{align*}

The pair $(\Omega(U,T),\partial\Omega(U,T))$ is the \emph{$(U,T)$-ruin}. For brevity, we sometimes write $(\Omega(U,T),\partial)$ to denote the $(U,T)$-ruin. Note that if $T=\emptyset$, then $\Omega(U,T)=\Sigma(U)$ and $\partial\Omega(U,T)=\emptyset$.

 For $s\in T$, set $U'=U-s$ and $T'=T-s$. As in \cite[Proof of Theorem 8.3]{DDJO}, we have the following weak exact sequence:

\[\begin{tikzcd}[column sep=small]
\cdot\cdot\cdot \arrow{r} & L_\Q^2H_{\ast}(\Omega(U',T'),\partial) \arrow{r} & L_\Q^2H_{\ast}(\Omega(U,T'),\partial) \arrow{r} & L_\Q^2H_{\ast}(\Omega(U,T),\partial) \arrow{r} &\cdot\cdot\cdot
\end{tikzcd}\]

For the special case when $U=S$ and $T=\{s\}$ the above sequence becomes:

\[\begin{tikzcd}[column sep=small]
\cdot\cdot\cdot \arrow{r} & L_\Q^2H_{\ast}(\Sigma(S-s)) \arrow{r} & L_\Q^2H_{\ast}(\Sigma(S)) \arrow{r} & L_\Q^2H_{\ast}(\Omega(S,T),\partial) \arrow{r} &\cdot\cdot\cdot
\end{tikzcd}\]

\begin{lemma}
\label{lemma:3mnfldnerveoneruin}
Suppose that $L$ is a flag triangulation of a $3$-manifold. Then for every $t\in L$, $L_\Q^2H_{\ast}(\Omega(S,t),\partial\Omega(S,t))=0$ for $\ast>2$ and $\Q\leq\mathbf{1}$.
\end{lemma}
\begin{proof}
First, for $t\in L$, note that the $(S,t)$-ruin has the property that

$$\Omega(S,t)=\Omega(St(t),t),$$

where $St(t)=\{s\in S\mid m_{st}<\infty\}$. Set $Lk(t)=St(t)-t$, and so we have the following weak exact sequence:

\[\begin{tikzcd}[column sep=small]
 \cdot\cdot\cdot\arrow{r} & L_\Q^2H_{\ast}(\Sigma(Lk(t))) \arrow{r} & L_\Q^2H_{\ast}(\Sigma(St(t))) \arrow{r} & L_\Q^2H_{\ast}(\Omega(S,t),\partial\Omega(S,t)) \arrow{r} &\cdot\cdot\cdot
\end{tikzcd}\]

Note that $$L_\Q^2b_{\ast}(\Sigma(St(t)))=L_\Q^2b_{\ast}(\Sigma_{St(t)}) \text{ and } L_\Q^2b_{\ast}(\Sigma(Lk(t)))=L_\Q^2b_{\ast}(\Sigma_{Lk(t)}),$$ where $\Sigma_{St(t)}$ and $\Sigma_{Lk(t)}$ are the Davis complexes corresponding to the subgroups $W_{St(t)}$ and $W_{Lk(t)}$, respectively. Since $L$ is flag, the respective nerves of the groups $W_{St(t)}$ and $W_{Lk(t)}$ are a $3$-disk and a $2$-sphere. Furthermore, the nerve of $W_{Lk(t)}$ is not the boundary of a $3$-simplex (again, $L$ is flag). By Theorem \ref{thm:weightedsingerfor34disks}, $L^2_\Q b_k(\Sigma_{St(t)})=0$ for $k> 2$, and by Theorem \ref{thm:weightedSingerdim3}, $L_\Q^2b_k(\Sigma_{Lk(t)})=0$ for $k>1$. Therefore weak exactness of the sequence implies that $L_\Q^2H_{\ast}(\Omega(S,t),\partial)=0$ for $\ast>2$.
\end{proof}

\begin{lemma}[Compare {\cite[Lemma 4.1]{Schroeder1}}]
\label{lemma:3mnfldnervetworuin}
For every $T\in\mathcal{S}^{(2)}$ and $U\subset S$ with $T\subset U$, we have $L_\Q^2H_4(\Omega(U,T),\partial\Omega(U,T))=0$ for $\Q\leq\mathbf{1}$.
\end{lemma}
\begin{proof}
The proof of \cite[Lemma 4.1]{Schroeder1} goes through to show that $L_\mathbf{1}^2H_4(\Omega(U,T),\partial)=0$, the main point being that $L$ is a flag triangulation of a $3$-manifold, and so it follows that $\Sigma_{cc}$ is a $4$-pseudomanifold (i.e. every $3$-cell of $\Sigma_L$ is contained in precisely two $4$-cells). The argument in \cite[Lemma 4.8]{Mogilski} now completes the proof.
\end{proof}

\begin{proof}[Proof of Theorem \ref{thm:weighted3manifoldnerve}]
With the above lemmas, we now follow \cite[Proof of The Main Theorem]{Schroeder1} line by line. For every $U\subset S$ and $t\in U$, we have the following weak exact sequence:
\[\begin{tikzcd}[column sep=small]
\cdot\cdot\cdot \arrow{r} & L_\Q^2H_{\ast}(\Sigma(U-t)) \arrow{r} & L_\Q^2H_{\ast}(\Sigma(U)) \arrow{r} & L_\Q^2H_{\ast}(\Omega(U,t),\partial) \arrow{r} &\cdot\cdot\cdot
\end{tikzcd}\]

By Lemma \ref{lemma:3mnfldnervetworuin} and \cite[Proposition 4.2]{Schroeder1}, $L_\Q^2H_{\ast}(\Omega(U,t),\partial)=0$ for $\ast>2$, and hence by exactness, $$L_\Q^2H_{\ast}(\Sigma(U-t))\cong L_\Q^2H_{\ast}(\Sigma(U)) \text{ for } \ast>2.$$ It follows that $L_\Q^2H_{\ast}(\Sigma(S))\cong L_\Q^2H_{\ast}(\Sigma(\emptyset))$ for $\ast>2$, and hence the theorem.
\end{proof}

\section*{Acknowledgements}
The author would like to thank his Ph.D. advisor Boris Okun for his invaluable guidance and discussions.

\end{document}